\documentclass[12pt]{amsart}

\usepackage{amssymb,amsbsy}
\usepackage{latexsym,array}
\usepackage{verbatim,color}
\usepackage{fullpage,euscript}

\usepackage[all]{xy}
\usepackage{xcolor}
\usepackage[colorlinks=true,linkcolor=blue,urlcolor=my_color,citecolor=magenta]{hyperref}

\definecolor{my_color}{rgb}{0,0.5,0.5}
\definecolor{MIXT}{rgb}{0.8,0.5,0.2}

\numberwithin{equation}{section}

\input {cyracc.def}
\font\tencyr=wncyr10 %scaled \magstephalf
\font\tencyi=wncyi10 %scaled \magstephalf
\font\tencysc=wncysc10 %scaled \magstephalf
\def\rus{\tencyr\cyracc}

\def\rusi{\tencyi\cyracc}
\def\rusc{\tencysc\cyracc}

\newtheorem{thm}{Theorem}[section]
\newtheorem{lm}[thm]{Lemma}%[chapter]
\newtheorem{cl}[thm]{Corollary}
\newtheorem{prop}[thm]{Proposition}

\newtheorem{qtn}{Question}

\theoremstyle{remark}
\newtheorem{rmk}[thm]{Remark}

\theoremstyle{definition}
\newtheorem{ex}[thm]{Example} %[section]
\newtheorem{opr}{Definition}

\theoremstyle{plain}

\newcommand {\be}{{\mathfrak b}}

\newcommand {\g}{{\mathfrak g}}

\newcommand {\h}{{\mathfrak h}}

\newcommand {\p}{{\mathfrak p}}
\newcommand {\q}{{\mathfrak q}}

\newcommand {\es}{{\mathfrak s}}
\newcommand {\te}{{\mathfrak t}}
\newcommand {\ut}{{\mathfrak u}}

\newcommand {\z}{{\mathfrak z}}
\newcommand {\sln}{{\mathfrak{sl}}_n}
\newcommand {\slv}{{\mathfrak{sl}}(\BV)}

\newcommand {\eus}{\EuScript}

\newcommand {\gN}{{\eus N}}

\newcommand {\gS}{{\eus S }}

\newcommand {\eps}{\epsilon}

\newcommand {\lb}{\lambda}

\newcommand {\ca}{{\mathcal A}}

\newcommand {\ck}{{\mathcal K}}

\newcommand {\N}{{\mathcal N}}
\newcommand {\co}{{\mathcal O}}

\newcommand {\BV}{{\mathbb V}}

\newcommand {\BZ}{{\mathbb Z}}
\newcommand {\BN}{{\mathbb N}}

\newcommand {\md}{/\!\!/}

\newcommand {\codim}{{\mathrm{codim\,}}}

\newcommand {\ind}{{\mathrm{ind\,}}}
\newcommand {\Lie}{{\mathsf{Lie\,}}}

\newcommand {\Ima}{{\mathrm{Im\,}}}
\newcommand {\Mor}{\operatorname{Mor}}

\newcommand {\rk}{{\mathrm{rk\,}}}

\newcommand {\spe}{{\mathsf{Spec\,}}}

\newcommand {\trdeg}{{\mathrm{tr.deg\,}}}

\newcommand {\GR}[2]{{\textrm{{\bf #1}}}_{#2}}
\renewcommand {\GR}[2]{{\mathsf{#1}}_{#2}}

\newcommand {\un}{\underline}

\newcommand {\qm}{{\q\langle m\rangle}}
\newcommand {\Qm}{{Q\langle m\rangle}}
\newcommand {\qq}{{\q\langle 1\rangle}}
\newcommand {\QQ}{{Q\langle 1\rangle}}
\newcommand {\NN}{{\N\langle 1\rangle}}
\newcommand {\NNN}{{\N\langle 1,1\rangle}}
\newcommand {\gNN}{{\gN\langle 1\rangle}}
\newcommand {\gNNN}{{\gN\langle 1,1\rangle}}
\newcommand {\gm}{{\g\langle m\rangle}}
\newcommand {\Gm}{{G\langle m\rangle}}
\newcommand {\go}{{\g\langle 1\rangle}}
\newcommand {\goo}{{\g\langle 1,1\rangle}}

\newcommand {\bxi}{{\boldsymbol{\xi}}}
\newcommand {\bb}{{\boldsymbol{b}}}
\newcommand {\bx}{\boldsymbol{x}}

\newcommand {\beq}{\begin{equation}}
\newcommand {\eeq}{\end{equation}}

\renewcommand{\le}{\leqslant}
\renewcommand{\ge}{\geqslant}

\newcommand {\bbk}{\Bbbk}%{\mathbb F}%

\begin{document}
\setlength{\parskip}{3pt plus 2pt minus 0pt}
\hfill { {\color{blue}\scriptsize October 9, 2017}}%%
\vskip1ex

\title%{}
{Takiff algebras with polynomial rings of symmetric invariants}
\author[D.\,Panyushev]{Dmitri I. Panyushev}
\address[D.\,Panyushev]%
{Institute for Information Transmission Problems of the R.A.S, Bolshoi Karetnyi per. 19, 
Moscow 127051, Russia}
\email{panyushev@iitp.ru}
\author[O.\,Yakimova]{Oksana S.~Yakimova}
\address[O.\,Yakimova]{Universit\"at zu K\"oln,
Mathematisches Institut, Weyertal 86-90, 50931 K\"oln, Deutschland}
\email{yakimova.oksana@uni-koeln.de}
\thanks{The research of the first author was carried out at the IITP RAS at the expense of the Russian Foundation for Sciences (project {\rus N0} 14-50-00150).  The second author is supported by  
the DFG (Heisenberg--Stipendium).}
\maketitle

%%%%%%%%%%%%%%%%%%%%%%%%%%
\section*{Introduction}

\noindent
The ground field $\bbk$ is algebraically closed and of characteristic $0$. Let $Q$ be a connected 
algebraic group with $\q=\Lie Q$ and $\gS(\q)=\bbk[\q^*]$ the symmetric algebra of $\q$.
The subalgebra of $Q$-invariants in $\bbk[\q^*]$ is denoted by $\bbk[\q^*]^Q$ or $\bbk[\q^*]^\q$. 
The elements of $\bbk[\q^*]^Q$ are called {\it symmetric invariants of\/} $\q$. 
Interesting classes of non-reductive groups $Q$ such that $\bbk[\q^*]^Q$ is a polynomial ring have 
recently been found, see e.g.~\cite{J,coadj,trio,alafe1,alafe2,MZ,Y}. A quest for this type of groups continues. 
Let $\qm:=\q\otimes \bbk[T]/(T^{m+1})$ be the {\it $m$-th Takiff algebra} (=\,a {\it truncated current 
algebra}) of $\q$. Since $\q\langle 0\rangle\simeq \q$, we may 
assume that $m\ge 1$. Our main result is that under a mild restriction, the passage from $\q$ to  $\qm$ 
preserves the polynomiality of symmetric invariants. We also 
{\it\bfseries (1)} discover a new phenomenon that a certain ideal of $\qm$ %, denoted $\qm^u$, 
has a polynomial ring of invariants in $\bbk[\qm^*]$, and
{\it\bfseries (2)} show that the property of $\q$ that $\bbk[\q^*]$ is a free  $\bbk[\q^*]^Q$-module does not
always extend to $\qq$.

The story began in 1971, when Takiff proved that if $\g$ is semisimple, then
$\g\langle 1\rangle=\g\ltimes\g^{\sf ab}$ has a polynomial ring of symmetric invariants whose Krull 
dimension equals $2{\cdot}\rk\g$~\cite{takiff}. Then Ra\"is and Tauvel proved a similar result for 
$\g\langle m\rangle$ with arbitrary $m\in \BN$~\cite{rt}. This %Ra\"is-Tauvel theorem 
is the classical analogue of the description of the  
Feigin-Frenkel centre $\z(\widehat{\g})\subset {\mathcal U}(t^{-1}\g[t^{-1}])$, see~\cite{ff}.
Recently, Macedo and Savage came up with 
a multi-parameter generalisation of the Ra\"is-Tauvel result. Namely, let
\beq      \label{eq:multi-car}
    \hat\g=\g\otimes \bbk[T_1,\dots,T_r]/(T_1^{m_1+1},\dots,T_r^{m_r+1})=:\g\langle m_1,\dots,m_r\rangle 
\eeq
be a {\it truncated multi-current algebra} of a {\bf semisimple} $\g$. Then $\bbk[\hat\g^*]^{\hat \g}$ is a 
polynomial ring of Krull dimension $(m_1+1)\dots (m_r+1){\cdot}\rk\g$, see~\cite{savage}. 
The proofs heavily use the fact that $\g$ is semisimple, when many structure results are available. For 
instance, both~\cite{rt} and \cite{savage} exploit Kostant's section for the set of the 
regular elements of $\g$.
%$e+\g_f\hookrightarrow \g_{\sf reg}$, where $\{e,h,f\}$ is a principal $\tri$-triple in $\g$. 
On the other hand, if $\g$ is simple and $\q=\g_e$ is the centraliser of a nilpotent element $e\in\g$
such that $\g_e$ has the ``{\sl codim}--$2$ property'' and $e$ admits a ``good generating system''
in $\bbk[\g]^G$, then $\bbk[\g_e\langle m\rangle^*]^{\g_e\langle m\rangle}$ is a polynomial ring for all
$m\in\BN$, see~\cite[Theorem\,3.1]{ap}.
In all these cases, the free generators of the ring of  symmetric invariants of $\hat\g$ or 
$\g_e\langle m\rangle$ are explicitly described via those of $\g$ or $\g_e$, respectively. This goes back to a general construction 
of~\cite{rt}. 

Our main theorem provides a substantial generalisation of all these partial results. 
To state it, we need some notation. The {\it index\/} of $\q$, $\ind\q$, is the minimal codimension of the $Q$-orbits in $\q^*$, hence $\ind\q=\rk\q$ if $\q$ is reductive.
Let $\textsl{d}f$ be the differential of $f\in \bbk[\q^*]$. We regard $\textsl{d}f$ as a polynomial mapping 
from $\q^*$ to $\q$ and write $(\textsl{d}f)_{\xi}$ for its value at $\xi\in\q^*$.  If $f\in\bbk[\q^*]^Q$, then
$\textsl{d}f$  is $Q$-equivariant. The image of 
$\q\otimes T+\dots +\q\otimes T^{m}$ in $\qm$ is an ideal  of codimension $\dim\q$, which is denoted 
by $\qm^u$. An open subset of an irreducible variety is called {\it big}, if its complement does not contain 
divisors. Then a brief version of our  result is

\begin{thm}   \label{thm:main}
Let $\q$ be an algebraic Lie algebra such that $\bbk[\q^*]^\q=\bbk[f_1,\dots,f_l]$ is a graded polynomial 
ring, where $l=\ind\q$. Set 
$\Omega_{\q^*}=\{\xi\in\q^*\mid (\textsl{d}f_1)_\xi\wedge\dots \wedge(\textsl{d}f_l)_\xi\ne 0\}$, and 
assume that  $\Omega_{\q^*}$ is big (in $\q^*$). For any $m\ge 1$, we then have
\begin{itemize}
\item[\sf (i)] \ $\bbk[\qm^*]^{\qm^u}$ is a graded polynomial ring of Krull dimension $\dim\q+ml$.
\item[\sf (ii)] \  the Takiff algebra $\q\langle m\rangle$ has the same properties as $\q$, i.e., 
$\bbk[\qm^*]^{\qm}$ is a graded polynomial ring of Krull dimension $(m+1)l=\ind\qm$ and the similarly 
defined subset $\Omega_{\qm^*}\subset\qm^*$ is also big.
\end{itemize}
\end{thm}
\noindent
(See also Theorem~\ref{thm:main1} for a description of free generators and $\Omega_{\qm^*}$.)
As is well-known, a semisimple Lie algebra $\g$ satisfies the assumptions of Theorem~\ref{thm:main}.
(This goes back to Chevalley and Kostant.)  
Therefore, Theorem~\ref{thm:main} yields another proof and a generalisation of~\cite[Theorem\,5.4]{savage}, 
see Corollary~\ref{cl:new-proof-gener}. A notable difference between our  Theorem~\ref{thm:main} and 
results of~\cite{ap} is that we do not impose a constraint on $\sum_i\deg f_i$, which is a part of the 
definition of a ``good generating system'', and do not require the {\sl codim}--$2$ property for $\q$ (see 
Section~\ref{sect:prelim} for the definition). A weaker assumption that $\Omega_{\q^*}$ is big appears to 
be sufficient. That is, our result applies to a larger supply of non-reductive Lie algebras, 
see examples in Sections~\ref{sect:prehom} and \ref{sect:ex-from-kot}. For instance, 
%using~\cite{SK}, we show that 
the {\it canonical truncation}, $\tilde\q$, of a Frobenius Lie algebra $\q$ satisfies the hypotheses of 
Theorem~\ref{thm:main}, see Section~\ref{subs:Frob}.

If $\g$ is {\bf semisimple}, then $\bbk[\gm^*]$ is a free $\bbk[\gm^*]^{\gm}$-module for any 
$m$~\cite[Appendix]{must}. In Section~\ref{sect:Eq}, we prove that this property does not generalise 
to the truncated multi-current algebras of $\g$ or the truncated current algebras $\qm$ for arbitrary $\q$ such 
that $\bbk[\qm^*]$ is a free $\bbk[\qm^*]^{\qm}$-module. Namely, $\bbk[\g\langle 1,1,1\rangle]$ is {\bf not} 
a free $\bbk[\g\langle 1,1,1\rangle]^{\g\langle 1,1,1\rangle}$-module (Theorem~\ref{thm:eq-multi}). This can also be interpreted as 
follows.
Since the passage $\g\leadsto \go$ preserve freeness of the module \cite{geof,must},
in the chain of Takiff extensions
\[
   \g\leadsto \go \leadsto  \go\langle 1\rangle \simeq \goo \leadsto \goo\langle 1\rangle \simeq 
   \g\langle 1,1,1\rangle,
\]
we loose the freeness of the module at the second or third step (conjecturally, at the third step!). This also implies that, for
$\g\langle 1,1,1,\dots,1\rangle=:\g\langle 1^r\rangle$ and every $r\ge 3$, 
$\bbk[\g\langle 1^r\rangle]$ is not a free module over the ring of symmetric invariants.

{\sl \un{ Notation}.}
Let $Q$ act on an irreducible affine variety $X$. Then $\bbk[X]^Q$ 
is the algebra of $Q$-invariant regular functions on $X$ and $\bbk(X)^Q$
is the field of $Q$-invariant rational functions. If $\bbk[X]^Q$ is finitely generated, then 
$X\md Q:=\spe \bbk[X]^Q$, and the {\it quotient morphism\/} $\pi_{Q}: X\to X\md Q$ is induced by
the inclusion $\bbk[X]^Q \hookrightarrow \bbk[X]$. If $\bbk[X]^Q$ is a graded polynomial ring, then
the elements of any set of algebraically independent homogeneous generators 
are called {\it basic invariants\/}. If $V$ is a $Q$-module and $v\in V$, then 
$\q_v=\{\zeta\in\q\mid \zeta{\cdot}v=0\}$ is the {\it stabiliser\/} of $v$ in $\q$ and 
$Q_v=\{s\in Q\mid s{\cdot}v=v\}$ is the {\it isotropy group\/} of $v$ in $Q$;  $H^o$ is the identity component of an algebraic group $H$.

%%%%%%%%%%%%%%%%%%%%%%%%%%%%%%%%
\section{Preliminaries on the coadjoint representation}
\label{sect:prelim}

\noindent
Let $Q$\/ be a connected affine algebraic group with Lie algebra $\q$. The symmetric algebra 
$\gS (\q)$ over $\bbk$ is identified with the graded algebra of polynomial functions on $\q^*$ and we also 
write $\bbk[\q^*]$ for it.  

The {\it index of}\/ $\q$, $\ind\q$, is the minimal codimension of $Q$-orbits in $\q^*$. Equivalently,
$\ind\q=\min_{\xi\in\q^*} \dim \q_\xi$. By Rosenlicht's theorem~\cite[2.3]{VP}, one also has
$\ind\q=\trdeg\bbk(\q^*)^Q$. The ``magic number'' associated with $\q$ is $\bb(\q)=(\dim\q+\ind\q)/2$. Since the coadjoint orbits are even-dimensional, the magic number is an integer. If $\q$ is reductive, then  
$\ind\q=\rk\q$ and $\bb(\q)$ equals the dimension of a Borel subalgebra. The Poisson bracket $\{\ ,\ \}$ in
$\bbk[\q^*]$ is defined on the elements of degree $1$ (i.e., on $\q$) by $\{x,y\}:=[x,y]$. 
The {\it centre\/} of the Poisson algebra $\gS(\q)$ is $\gS(\q)^\q=\{H\in \gS(\q)\mid \{H,x\}=0 \ \ \forall x\in\q\}$. Since $Q$ is connected, we also have $\gS(\q)^\q=\gS(\q)^{Q}=\bbk[\q^*]^Q$.

The set of $Q$-{\it regular\/} elements of $\q^*$ is 
$\q^*_{\sf reg}=\{\eta\in\q^*\mid \dim Q{\cdot}\eta\ge \dim Q{\cdot}\eta' \text{ for all }\eta'\in\q^*\}$.
We say that $\q$ has the {\sl codim}--$n$ property if $\codim (\q^*\setminus\q^*_{\sf reg})\ge n$.
The following useful result appears in~\cite[Theorem\,1.2]{coadj}:

\begin{thm}   \label{thm:coadj1.2}
Suppose that $\q$ has the {\sl codim}--$2$ property and there are homogeneous algebraically
independent $f_1,\dots,f_l\in \bbk[\q^*]^Q$ such that $l=\ind\q$ and $\sum_{i=1}^l \deg f_i=\bb(\q)$.
Then 
\begin{itemize}
\item[\sf (i)] \ $\bbk[\q^*]^Q=\bbk[f_1,\dots,f_l]$ and 
\item[\sf (ii)] \ $(\textsl{d}f_1)_\xi,\dots,(\textsl{d}f_l)_\xi$ are linearly
independent if and only if\/  $\xi\in\q^*_{\sf reg}$.
\end{itemize}
\end{thm}

\noindent 
Furthermore, if $\q$ has the {\sl codim}--$2$ property, then for any collection of algebraically independent 
homogeneous $f_1,\dots,f_l\in\bbk[\q]^Q$ with $l=\ind\q$, one has $\sum_{i=1}^l \deg f_i\ge \bb(\q)$.

\begin{opr}[cf. \cite{blms}]   \label{df:wonder}
An algebraic Lie algebra $\q$ is said to be $n$-{\it wonderful}, if
\begin{itemize}
\item[\sf (i)] \  $\q$ has the {\sl codim}--$n$ property.
\item[\sf (ii)] \  $\bbk[\q^*]^Q$ is a polynomial algebra of Krull dimension $l=\ind\q$;
\item[\sf (iii)] \ If $f_1,\ldots,f_l$ are basic invariants in $\bbk[\q^*]^Q$, then 
$\sum_{i=1}^l \deg f_i=\bb(\q)$.
\end{itemize} 
\end{opr}
\noindent
For instance, any semisimple Lie algebra is $3$-wonderful.

It follows from Theorem~\ref{thm:coadj1.2} that if $\q$ is 2-wonderful, then 
$\Omega_{\q^*}=\q^*_{\sf reg}$ is big. Therefore, Theorem~\ref{thm:main} applies to all 2-wonderful Lie 
algebras. (A more precise statement is given in Corollary~\ref{cor:main} below.)  For instance, it applies 
to all centralisers of nilpotent elements in types $\GR{A}{n}$ or $\GR{C}{n}$, 
see~\cite[Theorems\,4.2 \& 4.4]{trio} and~\cite[Section\,3]{ap}.

Suppose that $\bbk[\q^*]^Q$ is a polynomial ring, but nothing is known about the {\sl codim}--$2$ 
property.  Theorem~\ref{thm:main} suggests that one needs some tools to decide whether 
$\Omega_{\q^*}$ is big. In many cases, the following assertion is helpful. 

\begin{prop}[see {\cite[Prop.\,5.2]{js}}]   \label{prop:JS}
If\/ $\bbk[\q^*]^Q$ is a polynomial ring and $Q$ has no proper semi-invariants in $\bbk[\q^*]$, then
$\Omega_{\q^*}$ is big.
\end{prop}

\begin{rmk}    \label{rem:knop}
Using some ideas of Knop (see~\cite[Satz\,2]{kn86}), we can prove a more general assertion, which we do not need here. Namely, 
\\   \indent
{\it Let an algebraic group $Q$ act on an irreducible affine factorial variety $X$. Suppose that $X\md Q$ 
exists (i.e., $\bbk[X]^Q$ is finitely generated) and $\bbk[X]$ contains no proper $Q$-semi-invariants. Let 
$X_{\sf sm}$ denote the smooth locus of $X$ and $\pi_Q:X\to Y:=X\md Q$ the quotient morphism. Set 
$\Omega_X=\{x\in X_{\sf sm}\mid \pi_Q(x)\in Y_{\sf sm} \ \& \ (\textsl{d}\pi_Q)_x \ \text{ is onto}\}$. Then
$\Omega_X$ is big.
}
\end{rmk}

\section{Takiff algebras and their symmetric invariants}
\label{sect:Takiff-inv}

\noindent
By definition, the $m$-th Takiff algebra of $\q$ is $\qm:=\q\otimes \bbk[T]/(T^{m+1})$. In particular,
$\q\langle 1\rangle=\q\ltimes\q^{\sf ab}$ is the semi-direct product, where the second factor is an abelian 
ideal. For $j\le m$, the image of $\q\otimes T^j$ in $\qm$ is denoted by $\q_{[j]}$. 
A typical element of $\qm$ can be written as 
$\bx=(x_0, x_1,\dots,x_{m})$, where $x_j\in\q_{[j]}$. 
Likewise, we have $\qm^*\simeq \bigoplus_{j=0}^{m}(\q_{[j]}^*)$ as vector space, and %we write
$\bxi=(\xi_0,\xi_1,\dots,\xi_{m})$ is an element of $\qm^*$, where $\xi_j\in \q_{[j]}^*$. Then the pairing of 
$\qm$ and $\qm^*$ is given by
   ${<} \bx,\bxi{>}_{\qm}= \sum_{i=0}^{m}{<}x_i,\xi_i{>}_\q$.
It is sometimes convenient to regard the elements of $\qm$ and $\qm^*$ as  "polynomials'' in $\eps$, 
where $\eps^{m+1}=0$. Namely,
\[
  \bx_\eps%=x_0+\eps x_1+\dots +\eps^{m-1}x_{m-1}
  =\sum_{i=0}^{m}x_i\eps^{i} \quad\text{and} \quad   \bxi_\eps
  =\sum_{j=0}^{m}\xi_j\eps^{m-j} .
\]
Using this notation, the Lie bracket in $\qm$ is
\[
   [\bx_\eps,\boldsymbol{y}_\eps]=\sum_{0\le i+j\le m}[x_i,y_j]\eps^{i+j}
\]
and the coadjoint representation $\mathrm{ad}^*_{\qm}$ of $\qm$ is given by
\[
  (\mathrm{ad}^*_{\qm} \bx_\eps)\bxi_\eps=\sum_{0\le j-i\le m} \bigl(\mathrm{ad}^*_\q (x_i)\xi_j\bigr) \eps^{m-j+i}.
\]
Then $\qm^u=\bigoplus_{j=1}^{m}\q_{[j]}$ is an {\sf ad}-nilpotent ideal of $\qm$ and the corresponding 
connected algebraic group is 
\[
     \Qm\simeq Q\ltimes \exp(\qm^u)=Q\ltimes \Qm^u .
\]
(If $Q$ is reductive, then $\Qm^u$ is the unipotent radical of $\Qm$.) For a non-Abelian $Q$, the 
unipotent group $\Qm^u$ is commutative if and only if $m=1$.

By~\cite[2.8]{rt}, one has $\ind\qm=(m+1){\cdot} \ind\q$. Hence also $\bb(\qm)=(m+1){\cdot} \bb(\q)$. 
Moreover, 
\beq      \label{eq:qm-reg}
   \bxi\in \qm^*_{\sf reg} \ \text{ if and only if } \  \xi_{m}\in \q^*_{\sf reg} .
\eeq    
Therefore, the presence of {\sl codim}--$n$ property for $\q$ implies that for $\qm$.

A general method for constructing symmetric invariants of $\qm$ is presented in \cite{rt}. Suppose that 
$f\in \bbk[\q^*]$ is homogeneous. Recall that  $\textsl{d}f\in\Mor(\q^*,\q)$ is the differential of $f$. 
Consider $\bxi_\eps$ as an element of $\q^*\otimes\bbk[\eps]$
with $\eps^{m+1}=0$, and expand $f(\bxi_\eps)$ as a polynomial in $\eps$:
\[
   f(\xi_{m}+\eps\xi_{m-1}+\dots+\eps^{m-1}\xi_{1}+\eps^{m}\xi_0)=\sum_{j=0}^{m}F^j(\bxi)\eps^j .
\]
It is readily seen that $F^0(\bxi)=f(\xi_{m})$ and $F^1(\bxi)={<}(\textsl{d}f)_{\xi_{m}},\xi_{m-1}{>}_\q$. 
More generally, the following assertion is true.

\begin{prop}[see~\protect{\cite[Section\,III]{rt}}]      \label{prop:rt-1}
For any $j\in\{0,1,\dots,m\}$, we have 
\begin{itemize} 
\item[\sf (i)] \  $F^j(\bxi)={<}(\textsl{d}f)_{\xi_{m}},\xi_{m-j}{>}_\q + H_j(\xi_{m},\dots,\xi_{m-j+1})$ for some  $H_j\in \bbk[\qm^*]$;
\item[\sf (ii)] \ If $f\in \bbk[\q^*]^Q$, then every $F^j$ is a symmetric invariant of\/ $\qm$, i.e., $F^j\in \bbk[\qm^*]^{\Qm}$.
\end{itemize}
\end{prop}

Let $f_1,\dots,f_l$ be a set of basic invariants in $\bbk[\q^*]^Q$, where $l=\ind\q$. Using the above 
construction of~\cite{rt}, we associate to each $f_i$ the set of $\Qm$-invariants $F_i^0,\dots, F_i^{m}$.
Now, we are ready to state precisely our main result.

\begin{thm}   \label{thm:main1}
Let $Q$ be a connected algebraic group such that $\bbk[\q^*]^Q=\bbk[f_1,\dots,f_l]$ is a graded 
polynomial ring, where $l=\ind\q$. Set 
$\Omega_{\q^*}=\{\xi\in\q^*\mid (\textsl{d}f_1)_\xi\wedge\dots \wedge(\textsl{d}f_l)_\xi\ne 0\}$, and assume 
that $\Omega_{\q^*}$ is big. For any $m\ge 1$, we then have
\begin{itemize}
\item[\sf (i)] \ $\bbk[\qm^*]^{\Qm^u}$ is a graded polynomial ring of Krull dimension $\dim\q+ml$, which 
is freely generated by the coordinate functions on $\q_{[m]}^*$ and the $\{F_i^j\}$'s with 
$i=1,\dots,l$ and $j=1,\dots,m$.
\item[\sf (ii)] \  the Takiff algebra $\q\langle m\rangle$ has the same properties as $\q$, i.e., 
\begin{itemize}
\item $\bbk[\qm^*]^{\Qm}$ is a graded polynomial ring of Krull dimension $\ind\qm=(m+1)l$. (It is freely 
generated by the $\{F_i^j\}$'s with $i=1,\dots,l$ and $j=0,\dots,m$.)  
\item $\Omega_{\qm^*}=\bigoplus_{j=0}^{m-1}\q_{[j]}^*\times  \Omega_{\q^*}$ is big, where 
$\Omega_{\q^*}\subset \q_{[m]}^*\simeq \q^*$.
\end{itemize}
\end{itemize}
\end{thm}
\begin{proof} %[Proof of Theorem~\ref{thm:main}]   \leavevmode\par
{\sf (i)} \ Recall that $\qm\simeq \q\ltimes \qm^u$, where $\q=\q_{[0]}$ and 
$\qm^u=\bigoplus_{j=1}^{m}\q_{[j]}$, and $\Qm=Q\ltimes\Qm^u$. Here $\Qm^u$ is a unipotent 
normal subgroup of $\Qm$. 

Note that the subspace $\q_{[m]}\subset \qm$ regarded as a subset of $\bbk[\qm^*]$ belongs to the 
subalgebra of $\Qm^u$-invariants, and $F_i^0=f_i \in \gS[\q_{[m]}]$.
Let $\ca$ denote the subalgebra of $\bbk[\qm^*]^{\Qm^u}$ generated by $\q_{[m]}$ and
$\{F_i^j\}$ with $j=1,\dots,m$ and $i=1,\dots,l$. (Note that we do {\bf not} include $F_1^0,\dots,F_l^0$ in
the generating set for $\ca$!)

For $\bx=(x_0,\dots,x_{m})$ with $x_i\in\q_{[i]}$, we say that $x_j\ne 0$ is the {\it lowest component\/} of 
$\bx$, if $x_0=\dots =x_{j-1}=0$.  Now, $(\textsl{d}F_i^j)_\bxi\in\qm$ and using 
Proposition~\ref{prop:rt-1}(i), one readily verifies that its lowest component is 
$\bigl((\textsl{d}F_i^j)_\bxi\bigr)_{m-j}=(\textsl{d}f_i)_{\xi_{m}}\in \q_{[m-j]}$, where $j=0,1,\dots,m-1$.
Clearly, these lowest components are linearly independent if and only if $\xi_{m}\in\Omega_{\q^*}$. If 
$v_1,\dots,v_{\dim\q}$ is a basis for $\q_{[m]}$, then $(\textsl{d}v_i)_{\bxi}=v_i\in \q_{[m]}$.
Since all these differentials have a block-triangular form w.r.t. the decomposition 
$\qm=\bigoplus_{i=1}^{m}\q_{[i]}$ (cf. Table~\ref{table:differ}), it follows that the differentials {\it per se\/} are linearly independent at $\bxi$ if and only if 
$\xi_{m}\in\Omega_{\q^*}$.
Therefore, the polynomials 
\[
  v_1,\dots,v_{\dim\q}, \text{ and } \{F_i^j\} \ \text{ with } j=1,\dots,m,\ i=1,\dots,l
\]
are algebraically independent and generate $\ca$.
% and their differentials are linearly independent if $\xi_{m-1}\in \Omega_{\q^*} \subset \q_{[m-1]}^*\simeq\q^*$.
As the differentials of this family are linearly independent on the big open subset 
$\bigoplus_{j=0}^{m-1}\q_{[j]}^*\times  \Omega_{\q^*}$ of $\qm^*$, Theorem~1.1
in \cite{trio} guarantee us that $\ca$ is an algebraically closed subalgebra in $\bbk[\qm^*]$,
of Krull dimension $\dim\q+ml$.
\\ \indent
On the other hand, if $\bxi=(0,\dots,0,\xi_{m})$ and $\xi_{m}\in\q^*_{\sf reg}$, then $\dim \Qm^u{\cdot}\bxi=m(\dim\q-l)$. Hence
$\trdeg \bbk[\qm^*]^{\Qm^u}\le \dim\qm-\dim \Qm^u{\cdot}\bxi=
 \dim\q+ml$. Therefore $\ca\subset \bbk[\qm^*]^{\Qm^u}$ is an algebraic extension, which implies that $\ca = \bbk[\qm^*]^{\Qm^u}$.
In other words, $\bbk[\qm^*]^{\Qm^u}=\bbk[\q_{[m]}^*][F_i^j, i=1,\dots,l; j=1,\dots,m]$.

{\sf (ii)} \ Since $\Qm\simeq Q\ltimes \Qm^u$ and the $F_i^j$'s are already $\Qm$-invariant
(Prop.~\ref{prop:rt-1}(ii)), it follows from part 
{\sf (i)} that 
\begin{multline*}
\bbk[\qm^*]^\Qm=(\bbk[\qm^*]^{\Qm^u})^Q=\bbk[\q_{[m]}^*]^Q[\{F_i^j\}, 1\le i\le l; 1\le j\le m]\\
=\bbk[\{F_i^j\}, 1\le i\le l; 0\le j\le m].
\end{multline*}
Furthermore, the differentials of the total set of generators $\{F_i^j\}$, with the value $j=0$ included, are 
also linearly independent if and only if $\xi_{m}\in \Omega_{\q^*}\subset \q_{[m]}^*$, 
see~\cite[Lemma\,3.3]{rt} and Table~\ref{table:differ}. 
\begin{table}[htb]
\caption{Components of the differentials of basic invariants}   \label{table:differ}
\begin{center}
\begin{tabular}{c|c c c c c c|}
 &  $\q_{[m]}$ & $\q_{[m-1]}$ & $\q_{[m-2]}$ & \multicolumn{2}{c}{\dots\dots\dots} &  $\q_{[0]}$ \\ \hline
$(\textsl{d}F_1^0)_\bxi$  & $(\textsl{d}f_1)_{\xi_m}$  & 0 & \dots &\dots & \dots  &0 \\
%$(\textsl{d}F_2^0)_\bxi$  & $(\textsl{d}f_2)_{\xi_m}$  & 0 & \dots &\dots & \dots  &0 \\
\vdots & \vdots &\vdots & \dots &\dots &\dots  & 0 \\
$(\textsl{d}F_l^0)_\bxi$  & $(\textsl{d}f_l)_{\xi_m}$  & 0 & \dots &\dots & \dots  &0 \\
$(\textsl{d}F_1^1)_\bxi$  & $\ast$ & $(\textsl{d}f_1)_{\xi_m}$  & 0 & \dots &\dots  &0 \\
\vdots & \vdots &\vdots & \vdots &\dots &\dots &  0 \\
$(\textsl{d}F_l^1)_\bxi$ & $\ast$ & $(\textsl{d}f_l)_{\xi_m}$  & 0 &  \dots &\dots  &0 \\
$(\textsl{d}F_1^2)_\bxi$  & $\ast$ & $\ast$ & $(\textsl{d}f_1)_{\xi_m}$  & \dots &\dots  &0 \\
\vdots & \vdots &\vdots & \vdots &\dots &\dots &  0 \\
$(\textsl{d}F_l^2)_\bxi$ & $\ast$ & $\ast$ & $(\textsl{d}f_l)_{\xi_m}$  & \dots &\dots  &0 \\
\dots & \dots &\dots & \dots &\dots &\dots & \dots  \\
\end{tabular}
\end{center}
\end{table}
Therefore,
$\Omega_{\qm^*} = \bigoplus_{j=0}^{m-1}\q_{[j]}^*\times  \Omega_{\q^*}$ \  is big.
\end{proof}

\noindent
For future use, we record a by-product of the proof:
\begin{cl}
 $\bxi\in\Omega_{\qm^*} \Longleftrightarrow  \xi_m\in \Omega_{\q^*}$.
\end{cl}
\begin{rmk}
It appears that Theorem~\ref{thm:main1} is fully analogous to ~\cite[Theorem\,11.1]{p05}, where the polynomiality of invariants for the {\bf adjoint} representation of $\qm$ is studied.
\end{rmk}
\begin{cl}     \label{cor:main}
If\/ $\q$ is an $n$-wonderful algebra for $n\ge 2$, then so is $\qm$ for any $m\in\BN$. 
\end{cl}
\begin{proof} 
Let us check that the properties of Definition~\ref{df:wonder} carry over from $\q$ to $\qm$.

\textbullet \quad As noted above, the presence of {\sl codim}--$n$ property for $\q$ implies that for $\qm$.
We also have $\dim \qm=(m+1){\cdot}\dim\q$ and $\ind\qm=(m+1){\cdot}\ind\q$.

\textbullet \quad If $\q$ is 2-wonderful, then $(\textsl{d}f_1)_\xi,\dots,(\textsl{d}f_l)_\xi$ are linearly 
independent if and only if $\xi\in\q^*_{\sf reg}$ (Theorem~\ref{thm:coadj1.2}). Hence 
$\Omega_{\q^*}=\q^*_{\sf reg}$ and its complement does not contain divisors. Therefore, 
$\bbk[\qm^*]^\Qm$ is polynomial ring of Krull dimension $(m+1)l=(m+1){\cdot}\ind\q$, freely generated by 
the $F_i^j$'s.

\textbullet \quad Clearly, $\deg F_i^j=\deg f_i$ for all $i$ and $j$. Therefore \\
\[%\centerline{
\sum_{i=1}^l\sum_{j=0}^{m} \deg F_i^j =(m+1) \sum_{i=1}^l \deg f_i=(m+1) \bb(\q)=\bb(\qm). \qedhere
\]
\end{proof}

\begin{cl}[cf.~{\cite[Thm.\,5.4]{savage}}]   \label{cl:new-proof-gener}
For any $r$-tuple $m_1,\dots,m_r$, the truncated multi-current algebra $\q\langle m_1,\dots,m_r\rangle$
has a polynomial ring of symmetric invariants.
\end{cl}
\begin{proof}
A truncated multi-current algebra of any $\q$ is obtained as an iteration of various Takiff algebras. 
That is, 
\beq   \label{eq:multi-curr}
  \hat\q:=\q\langle m_1,\dots,m_r\rangle\simeq \Bigl(\dots\bigl((\q\langle m_1\rangle)\langle   
  m_2\rangle\bigr)\dots\Bigr)\langle m_r\rangle .
\eeq
Therefore, if $\q$ satisfies the hypotheses of Theorem~\ref{thm:main1}, then so is $\hat\q$. In particular,
$\bbk[\hat\q^*]^{\hat Q}$ is a polynomial ring. 
\end{proof}

Note that if $\q=\g$ is semisimple, then one can use %the Ra\"is--Tauvel 
results of~\cite{rt} only for the first iteration $\g\leadsto \g\langle m_1\rangle$, because 
afterwards the algebra in question becomes non-reductive.

\begin{rmk}    \label{rem:Igusa}
An essential point in our proof of Theorem~\ref{thm:main1} is the use of Theorem\,1.1 in~\cite{trio}. This 
ensures that the subalgebra $\ca$ is algebraically closed in $\bbk[\qm^*]$ and hence $\ca=\bbk[\qm^*]^{\Qm^u}$ for
the dimension reason. However, one can use instead an invariant-theoretic (geometric) argument related 
to Igusa's lemma (see e.g. \cite[Theorem\,4.12]{VP}) or ~\cite[Lemma\,6.1]{p05}).
Namely, consider the morphism 
\[
   \tau: \qm^* \to \q_{[m]}^*\times \mathbb A^{ml}=:Y
\]
given by $\tau(\bxi)=(\xi_{m}, F_1^1(\bxi),\dots,F_l^1(\bxi),\dots,F_1^{m}(\bxi),\dots,F_l^{m}(\bxi))$.
From the assumption on $\Omega_{\q^*}$ and a "triangular" form of $\{F_i^j\}$ (see~Prop.~\ref{prop:rt-1}(i)), one derives that 

(1)  \ $\Ima \tau\supset \Omega_{\q^*}\times \mathbb A^{ml}$, where the RHS is a big open subset of 
$Y$;

(2) \ for any $y\in \Omega_{\q^*}\times \mathbb A^{ml}$, the fibre $\tau^{-1}(y)$ is a sole $\Qm^u$-orbit.
\\
Then Igusa's lemma asserts that $\bbk[Y] \simeq \bbk[\qm^*]^{\Qm^u}$, i.e., $Y\simeq
\qm^*\md \Qm^u$ and $\tau=\pi_{\Qm^u}$. (Cf. the similar use of Igusa's lemma in~\cite[Theorems\,6.2 \& 11.1]{p05} and \cite[Theorem\,5.2]{coadj}.)
\end{rmk}

\section{Prehomogeneous vector spaces and rings of semi-invariants}
\label{sect:prehom}
\noindent
Here we show that some old results of Sato--Kimura~\cite{SK} on prehomogeneous vector spaces allow 
us to construct Lie algebras satisfying the hypotheses of Theorem~\ref{thm:main1}. 

\subsection{Prehomogeneous vector spaces}   \label{subs:SK}
Let $H\subset GL(V)$ be a representation of a connected group $H$ having an open orbit in 
$V$, i.e., $V$ is a {\it prehomogeneous vector space\/} w.r.t. $H$. By~\cite[\S\,4]{SK}, the algebra
of $H$-semi-invariants in $\bbk[V]$, denoted $\bbk[V]^{\langle H\rangle}$, is polynomial. More precisely,
let $\co\subset V$ be the open $H$-orbit and $D_1,\dots, D_l$ all simple divisors in $V\setminus\co$ (we do not need the irreducible components of codimension $\ge 2$ in $V$). If $D_i=\{f_i=0\}$, then
$f_i\in \bbk[V]^{\langle H\rangle}$, $f_1,\dots,f_l$ are algebraically independent, and 
$\bbk[V]^{\langle H\rangle}=\bbk[f_1,\dots,f_l]$. Moreover, let $\lb_i: H\to \bbk^\times$ be the $H$-character corresponding to $f_i$, i.e., $h{\cdot}f_i=\lb_i(h)f_i$ for all $h\in H$. Then the differentials of $\lb_i$'s are linearly independent and 
$\tilde H:=\{h\in H\mid \lb_i(h)=1 \ \ \forall i\}^o$ is of codimension $l$ in $H$. Then 
$[H,H]\subset \tilde H\subset H$ and
$\bbk[V]^{[H,H]}=\bbk[V]^{\tilde H}=\bbk[V]^{\langle H\rangle}$ is a polynomial ring.

\subsection{Frobenius Lie algebras}    \label{subs:Frob}
Suppose that $\ind\h=0$, i.e., $\h$ is {\it Frobenius}. Then $H$ has an open orbit in $\h^*$ and the above 
results apply to $V=\h^*$. Then $\bbk[\h^*]^{\langle H\rangle}=\bbk[\h^*]^{\tilde H}$ is a polynomial ring 
of Krull dimension $\dim H-\dim \tilde H=\ind\tilde\h$. Note that 
\[ \bbk[\tilde\h^*]=\gS(\tilde\h)\subset\gS(\h)=\bbk[\h^*],  \] 
and an important additional feature of the ``coadjoint'' situation is that 
$\bbk[\h^*]^{\tilde H}\subset \bbk[\tilde\h^*]$, see~\cite[Kap.\,II,\,\S\,6]{bgr}. Hence 
$\bbk[\h^*]^{\tilde H}=\bbk[\tilde\h^*]^{\tilde H}$, i.e., $\tilde\h$ has a polynomial ring of symmetric 
invariants whose Krull dimension equals $\ind\tilde\h$.  By the very construction, $\tilde H$ has no proper 
semi-invariants in $\bbk[\h^*]$ and hence in $\bbk[\tilde\h^*]$. It then follows from Proposition~\ref{prop:JS} that 
$\Omega_{\tilde\h^*}$ is big. Thus, Theorem~\ref{thm:main1} applies to $\tilde\h$, and hence 
$\tilde\h\langle m\rangle$ has a polynomial ring of symmetric invariants for any $m\ge 1$. 

\begin{rmk}   \label{rem:BGR}
More generally, for {\bf any} Lie algebra $\h$, the ring of symmetric semi-invariants 
$\bbk[\h^*]^{\langle H\rangle}$ (i.e., the {\it Poisson semi-centre\/} of $\gS(\h)=\bbk[\h^*]$) is isomorphic 
to the ring of symmetric invariants of a canonically defined subalgebra 
$\tilde\h\subset\h$~\cite[Kap.\,II,\,\S\,6]{bgr}, see also ~\cite[Sect.\,3]{fonya-et}. The subalgebra 
$\tilde\h$ is called the {\it canonical truncation\/} of $\h$. It has the property that
$\dim\h-\dim\tilde\h=\ind\tilde\h-\ind\h$~\cite[Lemma\,3.7]{fonya-et}, hence $\bb(\h)=\bb(\tilde\h)$.
Furthermore, since $\tilde H$ has no proper semi-invariants in $\bbk[\tilde\h^*]$, $\bbk(\tilde\h^*)^{\tilde H}$ is the field of fractions of $\bbk[\tilde\h^*]^{\tilde H}$ and the Krull dimension of $\bbk[\tilde\h^*]^{\tilde H}$ equals $\ind\tilde\h$.
Therefore, if $\bbk[\h^*]^{\langle H\rangle}=\bbk[\tilde\h^*]^{\tilde H}$ is a polynomial ring, then  
%and $\bbk[\tilde\h^*]$ has no proper $\tilde H$-semi-invariants. In this case,  
Proposition~\ref{prop:JS} and Theorem~\ref{thm:main1} apply 
to $\tilde\h$, and hence $\tilde\h\langle m\rangle$ has a polynomial ring of symmetric invariants for all $m\ge 1$. 
In the special case, where $\h$ is Frobenius, this is already explained in the previous paragraph.
\end{rmk}

Let us illustrate this theory in both Frobenius and non-Frobenius cases. 

\begin{ex}    \label{ex:b-ind=0}
Let $G$ be a simple algebraic group with $\Lie(G)=\g$, $\be$ a Borel subalgebra of $\g$, and 
$[\be,\be]=\ut$. The corresponding connected subgroups of $G$ are $B$ and $U$. Here we are 
interested in the symmetric invariants of $\be$, $\ut$, and the canonical truncation of $\be$. Most of 
these results are due to Kostant~\cite{kost} and Joseph~\cite{jos77}. (Actually, many Kostant's results 
are rather old and had been cited in \cite{jos77}.) Our idea is to demonstrate utility of the Sato--Kimura 
theory in this context.
\par
$({\color{MIXT}\lozenge_1})$ \ If $\ind \be=0$, then $\ind\ut=\rk\g$ and $\tilde\be=[\be,\be]=\ut$. Hence
$\gS(\be)^U=\gS(\ut)^U$ is a polynomial ring of Krull dimension $\rk\g$. As explained above, Theorem~\ref{thm:main1} applies to $\ut=\tilde\be$. 
It is well known that  $\ind\be=0$ {\sl if and only if\/} 
 $\g\in\{\GR{B}{n}, \GR{C}{n}, \GR{D}{2n}, \GR{E}{7}, \GR{E}{8}, \GR{F}{4}, \GR{G}{2}\}$.  

Let $f_1,\dots,f_{\rk\g}$ be the basic invariants in $\gS(\ut)^U$. Their weights and degrees are pointed out 
in~\cite[Tables\,I,II]{jos77}, with some corrections in \cite[Annexe\,A]{fj}. It follows from those data that 
$\sum_{i=1}^{\rk\g} \deg f_i < \bb(\ut)=\frac{1}{2}\dim\be$ unless $\g= \GR{C}{n}$. This means that, for all 
but one case, the {\sl codim}--$2$ property does not hold for $\ut$ (use 
Theorem~\ref{thm:coadj1.2}\,!). 
\par
$({\color{MIXT}\lozenge_2})$ \ If $\ind\be>0$, then $\ind\ut <\rk\g$ and $\gS(\ut)^U$ is a proper 
subalgebra of $\gS(\be)^U$. (Actually, one always has $\ind\ut+\ind\be=\rk\g$.) There are two possibilities to construct a suitable subalgebra of $\be$: one is related to the Sato--Kimura approach, and the other exploits the canonical truncation.

{\bf $\langle$-\,1\,-$\rangle$} 
Since $B$ has a dense orbit in $\ut^*$~\cite{kost}, one applies Sato--Kimura results to $V=\ut^*$, $H=B$, and $U=[B,B]$. This shows that $\gS(\ut)^U$ is still a polynomial ring. Moreover,
$\Omega_{\ut^*}$ is big for the same reason as above. 
For all these cases (i.e., $\g\in\{\GR{A}{n}, \GR{D}{2n+1}, \GR{E}{6}\}$), we have $\sum_i \deg f_i <\bb(\ut)$. 
Hence there is no {\sl codim}--$2$ property for $\ut$, but Theorem~\ref{thm:main1} applies to $\ut$.

{\bf $\langle$-\,2\,-$\rangle$} Now, the canonical truncation of $\be$  is a subalgebra that properly  
contains $\ut$. Namely, the toral part of $\tilde\be$ has dimension $\ind\be$. If $\be=\te\oplus\ut$ and 
$\Delta^+$ is the set of positive roots (=\,roots of $\ut$), then one canonically constructs the cascade 
$\ck$ of strongly orthogonal roots in $\Delta^+$ ({\it Kostant's cascade}), see \cite[Section\,2]{jos77}. If 
$\ck=\{\gamma_1,\dots,\gamma_t\}$, then $\ind\be=\dim\te-t$ and $\tilde\be=\tilde\te\oplus\ut$, where 
$\tilde\te=\{\gamma_1,\dots,\gamma_t\}^\perp$. Thus, we obtain that
\[
    \bbk[\be^*]^U=\bbk[\be^*]^{\langle B\rangle}=\bbk[\tilde\be^*]^{\tilde B} .
\]
By~\cite[4.16]{jos77},  $\gS(\be)^U$ is a polynomial ring of Krull dimension $\rk\g$. Hence 
Theorem~\ref{thm:main1} applies to $\tilde\be$.
\\ \indent
The output of this example is that, {\sl for any simple Lie algebra $\g$, our main theorem applies to both
$\tilde\be$ (the canonical truncation of $\be$) and\/ $\ut=[\be,\be]$. These two subalgebras of\/ $\be$ 
coincide if and only if\/ $\be$ is Frobenius.}
\end{ex}

\section{More examples}
\label{sect:ex-from-kot}
\noindent
We provide other applications of Theorem~\ref{thm:main1} to Lie algebras with or without the 
{\sl codim}--$2$ property.

\begin{ex}    \label{ex:sln}
Let $G\subset SL(V)$ be a representation of a connected semisimple algebraic group. Consider the 
semi-direct product  $\q=\g\ltimes V^{\sf ab}$. The corresponding connected group $Q=G\ltimes \exp(V)$ 
has no non-trivial characters, hence $\bbk[\q^*]$ does not contain proper $Q$-semi-invariants.
Therefore, if (we know that) $\bbk[\q^*]^Q$ is a polynomial ring, then $\Omega_{\q^*}$ is big (use 
Proposition~\ref{prop:JS}) and Theorem~\ref{thm:main1} applies to $\q$. The classification of 
representation $(G:V)$ of {\bf simple} algebraic groups $G$ such that $\bbk[\q^*]^Q$ is a polynomial ring 
is the subject of an ongoing project initiated by the second author. First non-trivial results for $G=SL_n$ 
are found in \cite{Y}, and the representations of the exceptional groups are considered 
in~\cite{kot-except}. The representations of $SO_n$ and $Sp_{2n}$ will be handled in our forthcoming 
publication.
   (However, it is not always easy to decide whether the {\sl codim}--$2$ property holds for such $\q$.)

Consider a concrete elementary example, where everything can be verified by hand. 
For an $n$-dimensional vector space $V$ with $n\ge 2$, take the semi-direct product
$\q=\slv\ltimes n\BV=\sln\ltimes n\bbk^n$. %where $n\BV$ is an abelian ideal in $\q$. 
The elements of $n\BV$ (resp. $n\BV^*$) are regarded as $n\times n$ matrices, where $\sln$ acts via left (resp. right)
multiplications. Since $\bbk[n\BV^*]^{SL(\BV)}=\bbk[\det]$ and generic stabilisers for the action
$(SL(\BV): n\BV)$ are trivial, we have
\[
     \bbk[\q^*]^Q=\bbk[n\BV^*]^{SL(\BV)}=\bbk[\det] .
\]
(The first equality here stems from \cite[Theorem\,6.4]{p05}.)
Hence $\ind\q=1$ and $\bb(\q)=n^2$. For an $n\times n$ matrix $\eta$, one has
$\textsl{d}(\det)_\eta= 0  \Leftrightarrow  \rk \eta <  n-1$. Therefore, $\textsl{d}(\det)$ vanishes on the
determinantal variety of matrices of rank $\le n-2$, which is of codimension $4$ in $n\BV^*$. 
Thus, $\Omega_{\q^*}$ is big.

On the other hand, $\q^*_{\sf reg}=\sln^*\times (n\BV)^*_{\det}$ is a principal open subset, i.e.,
$\q^*\setminus \q^*_{\sf reg}=\sln^*\times \{\det=0\}$ is a divisor. Hence the {\sl codim}--$2$ property 
does not hold here. This also follows from the fact that $n=\deg(\det) < \bb(\q)=n^2$.
\end{ex}

\begin{ex}   \label{ex:Heis-n}
Let $\mathfrak{Hei}_n$ be the Heisenberg Lie algebra of dimension $2n+1$. It has a basis 
$x_1,\dots,x_n,y_1,\dots,y_n,z$ such that the only nonzero brackets are $[x_i,y_i]=z$, $i=1,\dots,n$.
Then $\ind(\mathfrak{Hei}_n)=1$ and  $\bbk[\mathfrak{Hei}_n^*]^{\mathfrak{Hei}_n}=\bbk[z]$.
Therefore, $\Omega_{\mathfrak{Hei}_n}=\mathfrak{Hei}_n^*$ and Theorem~\ref{thm:main1} applies
here. It is easily seen that the hyperplane $\{\xi\in \mathfrak{Hei}_n^*\mid {<}\xi,z{>}=0\}$ consists of 
the fixed points of the Heisenberg group. Hence, $\mathfrak{Hei}_n$ does not have the {\sl codim}--$2$ 
property.

This has the following application to centralisers of nilpotent elements: 
\\ 
{\sl Let $G$ be a simple group of type $\GR{G}{2}$. If $G{\cdot}e\subset \g$ is the subregular nilpotent orbit, then $\dim\g_e=4$ and $\g_e\simeq \mathfrak{Hei}_1\oplus \bbk e$. }
\end{ex}

\begin{ex}   \label{CM-primery}
Let $e\in \g$ be nilpotent. Methods of~\cite{MZ} provide the polynomiality of $\bbk[\g_e^*]^{\g_e}$ for 
some nilpotent orbits that are not treated in \cite{trio}. In particular, Tables~2 and 3 in \cite{MZ} list such 
orbits for $G$ of type $\GR{E}{6}$ and $\GR{F}{4}$. For those of them, where the reductive part of 
$\g_e$ is semisimple, we know for sure that $\bbk[\g_e^*]$ has no proper $G_e^o$-semi-invariants, and 
hence Theorem~\ref{thm:main1} applies. Specifically, the four suitable $\GR{E}{6}$-orbits have the Dynkin-Bala-Carter labels $\GR{E}{6}(a_3), \GR{A}{5}, \GR{D}{4}, 2\GR{A}{2}+\GR{A}{1}$, whereas all
six $\GR{F}{4}$-orbits are suitable for us.
\end{ex}

\begin{ex}    \label{ex:borel-contr}
Associated with any parabolic subalgebra $\p$ of $\g$, there is an interesting contraction of $\g$, which is called a {\it parabolic contraction}, see~\cite{alafe2}. If $\p=\be$, then such a contraction has much better properties~\cite{alafe1}. Let $\be^-$ be an opposite Borel and $\ut^-=[\be^-,\be^-]$. Then
$\g=\be\oplus\ut^-$ is a vector space sum. The contraction in question is
$\q:=\be\ltimes (\ut^-)^{\sf ab}$, where $(\ut^-)^{\sf ab}$ is an abelian ideal of $\q$ and $(\ut^-)^{\sf ab}$
is regarded as $\be$-module via isomorphism $\g/\be \simeq \ut^-$. Note that $\q$ is solvable.

By~\cite[Section\,3]{alafe1}, we have {\sf (1)} \ $\ind\q=\rk\g$, {\sf (2)} \ $\bbk[\q^*]^Q$ is a 
polynomial ring, and {\sf (3)} \ the degrees of basic invariants are the same as those for $\g$. In particular, 
$\bb(\q)=\bb(\g)$ and if $f_1,\dots,f_l$ are the basic invariants in $\bbk[\q^*]^Q$, then 
$\sum_{i=1}^l \deg f_i=\bb(\q)$.

However, $\q$ does not have the {\sl codim}--$2$ property unless $\g$ is of type 
$\GR{A}{l}$~\cite[Theorem\,4.2]{alafe1}. Furthermore, $\Omega_{\q^*}$ is not big, if 
$\g\ne\GR{A}{l}$~\cite[Remark\,5.3]{contr}. Therefore, Theorem~\ref{thm:main1} does not apply  to
$\be\ltimes (\ut^-)^{\sf ab}$, if $\g\ne\GR{A}{l}$. But one can look at the canonical truncation of $\q$,
where the situation improves considerably. Following~\cite[Sect.\,5]{contr}, consider 
\[
   \tilde\q=\ut\ltimes (\ut^-)^{\sf ab} \subset \be\ltimes (\ut^-)^{\sf ab}=\q .
\]
Here one has $\tilde\q=[\q,\q]$, 
$\ind\tilde\q=\ind\q+ (\dim\q-\dim\tilde\q)=2\rk\g$, and hence $\bb(\tilde\q)=\bb(\q)=\bb(\g)$.
By~\cite[Theorem\,5.9]{contr},  $\gS(\tilde\q)^{\tilde\q}$ is a polynomial ring of Krull dimension $2\rk\g$.
The situation with the {\sl codim}--$2$ property for $\tilde\q$ remains the "same" as for $\q$, but 
$\Omega_{\tilde\q^*}$ is already a big open subset of $\tilde\q^*$ (see the proof of Theorem\,5.9 in \cite{contr}).  Thus,
Theorem~\ref{thm:main1} applies to $\tilde\q$ for {\bf all} simple $\g$.
\end{ex}

\begin{ex}    \label{ex:z2-contr}
Let $\g=\g_0\oplus\g_1$ be a $\BZ_2$-grading of a simple Lie algebra $\g$ and 
$\q=\g_0\ltimes\g_1^{\sf ab}$ the related $\BZ_2$-contraction. Then  $\ind\q=\rk\g$ and the 
{\sl codim}--$2$ property is always satisfied here (see~\cite{coadj}). Here $\g_0$ is reductive but not 
necessarily semisimple, and $\bbk[\q^*]^Q$ is a polynomial 
ring (in $\rk\g$ variables) if and only if the restriction homomorphism $\bbk[\g]^G \to \bbk[\g_1]^{G_0}$ is 
onto~\cite[Sect.\,6]{Y-imrn}. This excludes only four $\BZ_2$-gradinds related to the algebras of type $\GR{E}{n}$.
\end{ex}

\begin{ex}    \label{ex:seaweed}
Let $\p$ and $\p'$ be two parabolic subalgebras of $\g$ such that $\p+\p'=\g$. Then $\es=\p\cap\p'$ is
called a {\it seaweed} (or {\it biparabolic}) subalgebra of $\g$~\cite{seaweed}. By work of Joseph and his 
collaborators, it is known in many cases that $\bbk[\es^*]^{\langle S\rangle}$ is a polynomial ring. In 
particular, this is true for any $\es$, if $\g$ is of type $\GR{A}{n}$ or $\GR{C}{n}$~\cite{J}. (See also a summary of known results and other good cases in~\cite{FP}.)
Therefore, in all such good cases,
the canonical truncation of $\es$ (=\,{\it truncated biparabolic\/} in Joseph's terminology) is a good example for 
Theorem~\ref{thm:main1}.
\end{ex}

%%%%%%%%%%%%%%%%%%%%%%
\section{On the equidimensionality}
\label{sect:Eq}

\noindent
Whenever a connected algebraic group $Q$ has the property that $\bbk[\q^*]^Q$ is a polynomial ring, it 
is natural to inquire whether it is true that $\bbk[\q^*]$ is a free $\bbk[\q^*]^Q$-module.  The latter is
equivalent to that the enveloping algebra $\mathcal U(\q)$ is a free module over its centre 
$\mathcal Z(\q)\simeq \bbk[\q^*]^Q$. Assuming that $\bbk[\q^*]^Q$ is a polynomial ring, i.e., $\q^*\md Q$ is an affine space, the well-known 
geometric answer to this inquiry is that \\[.6ex]
\centerline{
$\bbk[\q^*]$ is a free $\bbk[\q^*]^Q$-module {\it if and only if\/} $\pi_Q:\q^*\to \q^*\md Q$ is equidimensional,}
\\[.6ex]
i.e., equivalently, the zero-fibre of $\pi_Q$, $\pi_Q^{-1}(\pi_Q(0))$, has the `right' dimension $\dim\q-\dim\q^*\md Q$. In the setting of Takiff algebras, one can raise the following: % question:

\begin{qtn} Suppose that the hypotheses of Theorem~\ref{thm:main1} hold for $\q$ and $\pi_Q$ is equidimensional. Is it true that 
$\pi_{\Qm}: \qm^*\to \qm^*\md\Qm$ is equidimensional, too?
\end{qtn}
\noindent
As we shall see below, the general answer to this question is ``no''.
The celebrated positive result is that if $\g$ is semisimple, then the zero-fibre of $\pi_{\Gm}$ 
is irreducible and $\pi_{\Gm}$ is equidimensional for any $m\in\BN$~\cite[Appendix]{must}. The reason
is that the usual nilpotent cone $\gN\subset\g\simeq\g^*$ is an irreducible complete intersection, and it 
has rational singularities. Here $\gN\langle m\rangle:=\pi_{\Gm}^{-1}(\pi_{\Gm}(0))$ is a jet scheme of 
$\gN$. 

For $m=1$, these results are obtained in~\cite{geof} via a case-by-case argument. (See also another 
approach and a generalisation in~\cite[Theorem\,10.2]{p05}.) 

In this section, we prove that the equidimensionality does 
not carry over to the multi-current setting, even for semisimple $\g$. 
Let $\hat\q=\q\langle m_1,\dots,m_r\rangle$ be a truncated multi-current algebra of $\q$, 
cf.~\eqref{eq:multi-car}. As in Section~\ref{sect:Takiff-inv}, we can write $\hat\q=
\bigoplus_{i_1,\dots,i_r}\q_{[i_1,\dots,i_r]}$ and likewise for $\hat\q^*$, where 
$0\le i_j\le m_j, \ j=1,\dots,r$. It then follows from \eqref{eq:qm-reg} and the iteration process~\eqref{eq:multi-curr} that 
\[
   \bxi=(\xi_{[i_1,\dots,i_r]})\in \hat\q^*_{\sf reg} \ \Longleftrightarrow \ \xi_{[m_1,\dots,m_r]}\in \q^*_{\sf reg} 
\]
(see also Prop.\,4.1(b) in \cite{savage}).
Assume that $\q$ satisfies all the assumptions of Theorem~\ref{thm:main1} and set
$\N=\pi_Q^{-1}(\pi_Q(0))\subset \q^*$. Then 
$\N\langle m_1,\dots,m_r\rangle\subset \hat\q^*$ stands for the zero-fibre of 
$\pi_{\hat Q}:\hat\q^*\to \hat\q^*\md\hat Q$. We work below with the case in which all $m_i=1$. Then 
$\hat\q$ is obtained as iteration of semi-direct products, the first step being 
$\q\leadsto \q\ltimes\q^{\sf ab}=\qq$. Let us investigate the relation between $\N$ and $\NN$.
This will also apply below to the passage from $\NN$ to $\NNN$.

Recall that $\bxi=(\xi_0,\xi_1)$ is an element of $\qq^*$.
If $\bbk[\q^*]^Q=\bbk[f_1,\dots,f_l]$ with $l=\ind\q$, then $\bbk[\qq^*]^{\QQ}$ is freely generated by $F_1^0,\dots,F_l^0,F_1^1,\dots,F_l^1$, where $F_i^0$ depends only on $\xi_1$ and
 $F_i^1(\xi_0,\xi_1)={<}(\textsl{d}f_i)_{\xi_{1}},\xi_{0}{>}_\q$. Therefore
 \beq      \label{eq:N1}
    \NN=\{(\xi_0,\xi_1) \mid \xi_1\in\N \ \& \ {<}(\textsl{d}f_i)_{\xi_{1}},\xi_{0}{>}_\q=0 \ \forall i\}.
 \eeq
Since $\textsl{d}f_i$ is a $Q$-equivariant morphism from $\q^*$ to $\q$, we have 
$(\textsl{d}f_i)_{\xi}\in \q_{\xi}$. Moreover, if $\xi\in\q^*_{\sf reg}\cap \Omega_{\q^*}$, then
$\{(\textsl{d}f_1)_{\xi},\dots,(\textsl{d}f_l)_{\xi}\}$ is a basis for $\q_\xi$.
Consider the stratification of $\N$ determined by the basic invariants $f_1,\dots,f_l$. Set
\[
   \mathfrak X_{i,\N}=\{\xi\in\N\mid 
   \dim \mathsf{span}\bigl(\{(\textsl{d}f_1)_{\xi},\dots,(\textsl{d}f_l)_{\xi}\}\bigr)\le i\}.
\]
Then $\{0\}=\mathfrak X_{0,\N}\subset \mathfrak X_{1,\N}\subset\dots \subset \mathfrak X_{l,\N}=\N$. If
$\N=\bigcup_j\N_j$ is the irreducible decomposition, then $\mathfrak X_{i,\N_j}$ is similarly defined for 
any $j$. Set ${\mathfrak X}^o_{i,\N_j}=\mathfrak X_{i,\N_j}
\setminus \mathfrak X_{i-1,\N_j}$ for $i>0$ and ${\mathfrak X}^o_{0,\N_j}=\{0\}$. Clearly, each
${\mathfrak X}^o_{i,\N_j}$ is irreducible and open in ${\mathfrak X}_{i,\N_j}$. However,
${\mathfrak X}^o_{i,\N_j}$ can be empty for some $i,j$.
It follows from \eqref{eq:N1} that
$p:  \NN\to \N$, $(\xi_0,\xi_1)\mapsto \xi_1$, is a surjective projection and
\[
  \dim p^{-1}({\mathfrak X}^o_{i,\N_j})=\dim {\mathfrak X}^o_{i,\N_j}+\dim\q-i .
\]
Since $\qq$ has a polynomial ring of symmetric invariants, with $2l$ basic invariants
$F_1^0\dots,F_l^0,F_1^1,\dots,F_l^1$, one can consider the corresponding stratification of
$\NN$: \ 
\[
  \{0\}=\mathfrak X_{0,\NN}\subset \mathfrak X_{1,\NN}\subset\dots \subset \mathfrak X_{2l,\NN}=\NN .
\]
\begin{lm}   \label{lm:2-stratific}
We have  $\displaystyle p^{-1}({\mathfrak X}^o_{i,\N}) \subset 
\bigcup_{j=2i}^{l+i} {\mathfrak X}^o_{j,\NN}$.
\end{lm}
\begin{proof}
By definition, $\dim \mathsf{span}\bigl(\{(\textsl{d}f_1)_{\xi},\dots,(\textsl{d}f_l)_{\xi}\}\bigr)=i$ for
$\xi\in {\mathfrak X}^o_{i,\N}$. This clearly implies that, for $\bxi=(\xi_0,\xi)\in p^{-1}(\xi)$, we have
$\dim \mathsf{span}\bigl(\{(\textsl{d}F_1^0)_{\bxi},\dots,(\textsl{d}F_l^0)_{\bxi}\}\bigr)= i$ and 
$\dim \mathsf{span}\bigl(\{(\textsl{d}F_1^1)_{\bxi},\dots,(\textsl{d}F_l^1)_{\bxi}\}\bigr)\ge i$ (cf. Table~\ref{table:differ} with $m=1$). Furthermore, the lowest components of $(\textsl{d}F_j^0)_{\bxi}$
and $(\textsl{d}F_j^1)_{\bxi}$ belong to different graded pieces of $\qq$.
\end{proof}
By Lemma~\ref{lm:2-stratific}, the closures of $p^{-1}({\mathfrak X}^o_{l,\N_j})$ for all $j$ are the only subvarieties of $\NN$ that meet $\Omega_{\qq^*}$. Therefore, if 
${\mathfrak X}^o_{l,\N_j}\ne\varnothing$, then $\overline{p^{-1}({\mathfrak X}^o_{l,\N_j})}$ is an irreducible component of $\NN$ of dimension $\dim\N_j+\dim\q-l$.
Since $\dim\N_j\ge \dim\q-l$ for all $j$, one readily obtains

\begin{prop}  \label{prop:estimate}
If\/ $\q$ satisfies all the assumptions of Theorem~\ref{thm:main1}, then the following two
conditions are equivalent:
\begin{enumerate}
\item \ $\pi_\QQ$ is equidimensional, i.e., $\dim\NN=\dim\qq-\ind\qq=2(\dim\q-l)$;
\item 
\begin{itemize}
\item[\sf (i)]  \ $\pi_Q$ is equidimensional, i.e., $\dim\N_j=\dim\q-l$ for all $j$;
\item[\sf (ii)] \  ${\mathfrak X}^o_{l,\N_j}\ne \varnothing$ for all $j$ (i.e., 
$\N_j\cap\Omega_{\q^*}\ne \varnothing$);
\item[\sf (iii)] \ $\codim_{\N_j}({\mathfrak X}^o_{i,\N_j})\ge l-i$ for $i<l$.
\end{itemize}
\end{enumerate}
\end{prop}

This yields a sufficient condition for the absence of equidimensionality of $\pi_\QQ$:

\begin{cl}     \label{cl:not-eq}
If there is an irreducible component $\N_j$ of $\N$ such that $\N_j\cap\Omega_{\q^*}= \varnothing$,
then $\dim p^{-1}(\N_j)>2(\dim\q-l)$. Hence $\pi_\QQ$ is not equidimensional. 
\end{cl}
\noindent
We say that such $\N_j$ is a {\it bad\/} irreducible component of $\N$.

\begin{rmk}
If $\N$ is irreducible and $\dim\N=\dim\q-l$, then a similar analysis shows that $\NN$ is irreducible 
{\sl if and only if\/} conditions {\sf (i),\,(ii),} and {\sf (iii)'} hold, where {\sf (i),\,(ii)} are as above, with $\N$ in place of $\N_j$, and the last one is a bit stronger than {\sf (iii)}:
\\
\centerline{ {\sf (iii)'} \ $\codim_{\N}({\mathfrak X}^o_{i,\N})> l-i$ for $i<l$.}
\\
For, the closure of $p^{-1}(\overset{o}{\mathfrak X_{l,\N}})$ is always an irreducible component of $\NN$ of the `right' dimension $2(\dim\q-l)$, and we need the condition that $p^{-1}({\mathfrak X}^o_{i,\N})$ does not yield another component, i.e.,
$\dim p^{-1}({\mathfrak X}^o_{i,\N})<
2(\dim\q-l)$ for $i<l$.
\end{rmk}
 
From now on, we assume that $\q=\g$ is a simple Lie algebra of rank $l$.
Let us recall some properties of the nilpotent cone $\gN\subset\g^*\simeq \g$:
\begin{itemize}
\item \ $\gN$ is irreducible and contains finitely many $G$-orbits;
\item \ ${\mathfrak X}^o_{l,\gN}$ is the {\it principal} (or, {\it regular}) nilpotent orbit;
\item \ ${\mathfrak X}_{l-1,\gN}$ is irreducible of dimension $\dim\gN-2$ and 
${\mathfrak X}^o_{l-1,\gN}\ne \varnothing$  (it contains the {\it subregular} nilpotent orbit as a dense 
open subset). Moreover, if $\deg f_1\le \dots \le \deg f_l$, then 
$\deg f_{l-1}< \deg f_l$ and $(\textsl{d}f_l)_\xi=0$ for all $\xi\in {\mathfrak X}_{l-1,\gN}$.
\end{itemize}

Then $\gNN$ is also irreducible, and for the projection $p:\gNN\to \gN$, we have:
\begin{itemize}
\item[{\bf --}] \ $p^{-1}({\mathfrak X}^o_{l,\gN})$ is the open dense $G\langle 1\rangle$-orbit in 
$\gNN$, of dimension $2(\dim\g-l)$;
\item[{\bf --}]  \ $\dim p^{-1}({\mathfrak X}^o_{l-1,\gN})=2(\dim\g-l)-1$. Hence the closure of 
$p^{-1}({\mathfrak X}^o_{l-1,\gN})$ is a simple divisor, say $D$, in $\gNN$.
By Lemma~\ref{lm:2-stratific}, $p^{-1}({\mathfrak X}^o_{l-1,\gN})\subset 
{\mathfrak X}^o_{2l-2,\gNN}\cup {\mathfrak X}^o_{2l-1,\gNN}$.
\end{itemize}
\noindent
The next iteration replaces $\go$ with  $\goo\simeq\go \ltimes \go^{\sf ab}$
and provides the surjective projection $p_{1}:\gNNN\simeq \gNN\langle 1\rangle \to \gNN$. 
Here we are interested in $p_1^{-1}(D)$.
There is a dichotomy: either {\sf (1)} \ $D\cap{\mathfrak X}^o_{2l-1,\gNN}\ne \varnothing$ or \ 
{\sf (2)} \ $D\subset {\mathfrak X}_{2l-2,\gNN}$. \par
\textbullet \ \ In the first case, $\dim p_1^{-1}(D)=4(\dim\g-l)$ and it is an irreducible component of $\gNNN$ that is different from the closure of 
$p_1^{-1}({\mathfrak X}^o_{2l,\gNN})$. In other words, $p_1^{-1}(D)$ is a bad irreducible
component of $\gNNN$. Hence $\pi_{\g\langle 1,1,1\rangle}$ is not equidimensional by 
Corollary~\ref{cl:not-eq}. \par
\textbullet \ \ In the second case, $\dim p_1^{-1}(D)=4(\dim\g-l)+1$. Hence 
$\pi_{\g\langle 1,1\rangle}$ is already not equidimensional and then $\pi_{\g\langle 1,1,1\rangle}$ is not 
equidimensional, too.

Thus, we have proved

\begin{thm}  \label{thm:eq-multi}
Let $\g$ be a simple Lie algebra. Then 
\begin{itemize}
\item[\sf (i)] \ $\gNNN\subset \g\langle 1,1\rangle^*$ is reducible;
\item[\sf (ii)] \ $\gN\langle 1,1,1\rangle\subset \g\langle 1,1,1\rangle^*$ is not pure, i.e.,
$\pi_{\g\langle 1,1,1\rangle}$ is not equidimensional.
\end{itemize}
\end{thm}
\noindent

\begin{rmk}
If $\g=\g_1\oplus\dots\oplus\g_s$ is semisimple, where each $\g_i$ is simple and $s\ge 2$, then
$\g\langle m_1,\dots,m_r\rangle \simeq \bigoplus_{i=1}^s \g_i\langle m_1,\dots,m_r\rangle$. Therefore,
Theorem~\ref{thm:eq-multi} readily extends to the semisimple case.
\end{rmk}
Although $\gNNN$ is reducible, it still might be true that $\pi_{\g\langle 1,1\rangle}$ is equidimensional; in 
particular, it is likely that the second case above does not materialise. In fact, we hope (conjecture) that 
$\pi_{\g\langle 1,1\rangle}$ is always equidimensional.

\end{document}